\documentclass[ECP]{ejpecp0} % replace ECP by EJP if needed.  

\TITLE{The impact of selection in the $\Lambda$-Wright-Fisher model \footnote{This article was originally published in ECP 72. Vol.18.(2013). This updated version slighlty generalizes an argument and fills a gap in the proof of the last lemma (Erratum published in ECP 15. Vol.19.(2014)).}} % \thanks is optional. Insert line breaks with \\

\AUTHORS{%
  Cl\'ement Foucart\footnote{Universit\'e Paris 13, Paris, France.
    \EMAIL{foucart@math.univ-paris13.fr}}}%AUTHORS
 % Separate items with ;

\AMSSUBJ{60J25; 60J75; 60J28; 60G09; 92D25; 92D15} % Edit. Separate items with ;
\VOLUME{18}
\YEAR{2013}
\PAPERNUM{72}
\DOI{vVOL-PID}

\ABSTRACT{The purpose of this article is to study some asymptotic properties of the $\Lambda$-Wright-Fisher process with selection. This process represents the frequency of a disadvantaged allele. The resampling mechanism is governed by a finite measure $\Lambda$ on $[0,1]$ and selection by a parameter $\alpha$. When the measure $\Lambda$ obeys $\int_{0}^{1}-\log(1-x)x^{-2}\Lambda(dx)<\infty$, some particular behaviour in the frequency of the allele can occur. The selection coefficient $\alpha$ may be large enough to override the random genetic drift. In other words, for certain selection pressure, the disadvantaged allele will vanish asymptotically with probability one. This phenomenon cannot occur in the classical Wright-Fisher diffusion. We study the dual process of the $\Lambda$-Wright-Fisher process with selection and prove this result through martingale arguments.}

\begin{document}

%%%%%%%%%%%%%%%%%%%%%%%%%%%%%%%%%%%%%%%%%%%%%%%%%%%%%%%%%%%%%%%%%%%
%%                                                               %%
%% No need for \maketitle.                                       %%
%%                                                               %%
%%%%%%%%%%%%%%%%%%%%%%%%%%%%%%%%%%%%%%%%%%%%%%%%%%%%%%%%%%%%%%%%%%%

%%%%%%%%%%%%%%%%%%%%%%%%%%%%%%%%%%%%%%%%%%%%%%%%%%%%%%%%%%%%%%%%%%%
%%                                                               %%
%% Please replace what follows by the body of your article       %%
%% (up to the bibliography):                                     %%
%%                                                               %%
%%%%%%%%%%%%%%%%%%%%%%%%%%%%%%%%%%%%%%%%%%%%%%%%%%%%%%%%%%%%%%%%%%%

\section{Introduction and main result}
We recall here the basics about the $\Lambda$-Wright-Fisher process with selection. This process represents the evolution of the frequency of a deleterious allele. When no selection is taken into account, we refer the reader to Bertoin-Le Gall \cite{LGB2} and Dawson-Li \cite{Dawson} who have introduced this process as a solution to some specific stochastic differential equation driven by a random Poisson measure. Recently Bah and Pardoux \cite{Bah} have considered a lookdown approach to construct a particle system whose empirical distribution converges to the strong solution to
\begin{equation}\label{SDE}
X_{t}=x+\int_{[0,t]\times [0,1] \times [0,1]}z\left(1_{u\leq X_{s-}}-X_{s-}\right)\bar{\mathcal{M}}(ds,du,dz)-\alpha\int_{0}^{t}X_{s}(1-X_{s})ds
\end{equation}
where $\bar{\mathcal{M}}$ is a compensated Poisson measure $\mathcal{M}$ on $\mathbb{R}_{+}\times [0,1]\times [0,1]$ whose intensity is $ds\otimes du\otimes z^{-2}\Lambda(dz)$. Strong uniqueness of the solution to (\ref{SDE}) follows from an application of Theorem 2.1 in \cite{Dawson}. The process $(X_{t}, t\geq 0)$ should be interpreted as follows: it represents the frequency of a deleterious allele as time passes. When $\alpha>0$, the logistic term $-\alpha X_{t}(1-X_{t})dt$ makes the frequency of the allele decrease, this is the phenomenon of selection. Heuristically, the equation (\ref{SDE}) can be understood as follows:
\begin{itemize}
\item Denote the frequency of the allele just before time $s$ by $X_{s-}$. If $(s,u,z)$ is an atom of the measure $\mathcal{M}$, then, at time $s$,  
\begin{itemize}
\item if $u\leq X_{s-}$, the frequency of the allele increases by a fraction $z(1-X_{s-})$
\item if $u>X_{s-}$, the frequency of the allele decreases by a fraction $zX_{s-}$.   
\end{itemize}
\item Continuously in time, the frequency decreases due to the deterministic selection mechanism.
\end{itemize}
Note that we are dealing with a two-allele model: at any time $t$, the \textit{advantageous} allele has frequency $1-X_{t}$. The purely diffusive case is well understood (this is the classical Wright-Fisher diffusion, see e.g. Chapters 3 and 5 of Etheridge's monography \cite{etheridge2011} for a complete study). We mention that Section 5 of Bah and Pardoux \cite{Bah} incorporates a diffusion term in the SDE (\ref{SDE}). In such cases, the measure $\Lambda$ has an atom at $0$ and it has been already established in \cite{Bah} that these processes are absorbed in finite time. We then focus on measures $\Lambda$ carried on $]0,1]$ (see Remark \ref{rem}). Lastly, the process $(X_{t}, t\geq 0)$ should be interpreted as one of the simplest models introducing natural selection together with random genetic drift (that is, the random resampling governed by $\Lambda$).
\\

Plainly, the process $(X_{t}, t\geq 0)$ lies in $[0,1]$ and is a supermartingale. Therefore, the process $(X_{t}, t\geq 0)$ has an almost-sure limit denoted by $X_{\infty}$. This random variable is the frequency at equilibrium. Since $0$ and $1$ are the only absorbing states, the random variable $X_{\infty}$ lies in $\{0,1\}$. Moreover if $\alpha>0$, the supermartingale property yields that for all $x$ in $[0,1]$, \[\mathbb{P}[X_{\infty}=1|X_{0}=x]=\mathbb{E}[X_{\infty}|X_{0}=x]<x.\]
Our main result is the following theorem.
\begin{theorem}\label{main}Let $\alpha^{\star}:=-\int_{0}^{1}\log(1-x)\frac{\Lambda(dx)}{x^{2}}\in (0,\infty]$. Then,
\begin{itemize}
%\item[1)] If $\int_{0}^{1}x^{-1}\Lambda(dx)=\infty$ (dust-free condition), for all selection parameter $\alpha\geq 0$ and all $x\in (0,1)$, then $0<\mathbb{P}_{x}[X_{\infty}=0]<1$: the deleterious allele does not disappear almost surely.
%\begin{itemize} 
\item[1)] if $\alpha<\alpha^{\star}$ then for all $x\in (0,1)$, $0<\mathbb{P}[X_{\infty}=0|X_{0}=x]<1$,
\item[2)] if $\alpha^{\star}<\infty$ and $\alpha>\alpha^{\star}$ then $X_{\infty}=0$ a.s.
%\begin{center} $X_{\infty}=0$ a.s. if and only if $\alpha\geq \alpha^{\star}$.\end{center}
%\begin{itemize}
%\item[2-i)] if $\alpha<\alpha^{\star}$ then $0<\mathbb{P}[X_{\infty}=0]<1$,
%\item[2-ii)] if $\alpha>\alpha^{\star}$ then $X_{\infty}=0$ a.s.
%\end{itemize} 
%\end{itemize} 
\end{itemize}
\end{theorem} 
\begin{remark}
\begin{itemize}
\item As already mentioned, some $\Lambda$-Wright-Fisher processes with selection are absorbed in finite time (for instance the diffusive one). Such processes verify $\alpha^{\star}=\infty$. More precisely, Bah and Pardoux in Section 4.2 of \cite{Bah} show that they are related to measures $\Lambda$ satisfying the criterion of coming down from infinity.    
\item The condition $\int_{0}^{1}x^{-1}\Lambda(dx)=\infty$ implies that $-\int_{0}^{1}\log(1-x)x^{-2}\Lambda(dx)=\infty$. One can recognize the first integral condition as the dust-free criterion (see Lemma 25 and Proposition 26 in Pitman's article \cite{Pitman}). In other words, the dust-free condition ensures that the deleterious allele does not disappear with probability one. Namely, it may survive in the long run with positive probability. It is worth observing that some measure $\Lambda$ verify $-\int_{0}^{1}\log(1-x)x^{-2}\Lambda(dx)=\infty$ and $\int_{0}^{1}x^{-1}\Lambda(dx)<\infty$. An example is provided in the proof of Corollary 4.2 of M\"ohle and Herriger \cite{MohleHerriger}.
\item Bah and Pardoux in Section 4.3 of \cite{Bah} have obtained a first result on the impact of selection. Namely they show that if $\alpha>\mu:=\int_{0}^{1}\frac{1}{x(1-x)}\Lambda(dx)$ then $X_{\infty}=0$ almost surely. We highlight that the quantity $\mu$ is strictly larger than $\alpha^{\star}$ and that our method does not rely on the look-down construction. 
\item Der, Epstein and Plotkin \cite{Der} and \cite{Der1} obtain several results in the framework of finite populations with selection. They announce the results of Theorem \ref{main} in \cite{Der1}. However their proofs treat only the case when $\Lambda$ is a Dirac mass. Their method is based on a study of the generator of $(X_{t}, t\geq 0)$ and differs from ours.
\end{itemize}
\end{remark}
Except in the case of simple measures $\Lambda$, the expression of $\alpha^{\star}$ is rather complicated. We provide a few examples.
\begin{example} 
\begin{itemize}
\item Let $x \in ]0,1]$ and $c>0$, consider $\Lambda=c\delta_{x}$. We have \[\alpha^{\star}:x \mapsto -c\log(1-x)/x^{2}.\] The limit case $x=0$ corresponds to the Wright-Fisher diffusion and we have $\alpha^{\star}(0)=\infty$. When $x=1$, we also have $\alpha^{\star}(1)=\infty$ (this is the so-called star-shaped mechanism). Note that the map $\alpha^{\star}$ is convex and has a local minimum in $(0,1)$. Thus, in this model (called the Eldon-Wakeley model, see e.g. Birkner and Blath \cite{MR2562160}) the selection pressure which ensures the extinction of the disadvantaged allele is not a monotonic function of $x$. 
\item Let $a>0, b>0$, consider $\Lambda=Beta(a,b)$ where $Beta(a,b)$ is the unnormalized Beta measure with density $f(x)=x^{a-1}(1-x)^{b-1}$. 
\begin{itemize}
\item If $a=2$, one can easily compute $\alpha^{\star}(b)= \int_{0}^{\infty}\frac{te^{-bt}}{1-e^{-t}}dt=\zeta(2,b)$ (where $\zeta$ denotes the Hurwitz Zeta function).
\item If $b=1$ and $a>1$, we have $\alpha^{\star}(a)=\int_{0}^{\infty}te^{-t}(1-e^{-t})^{a-3}dt$.  If $a\leq 1, \alpha^{\star}(a)=\infty$.       
\end{itemize} 
The computation is more involved for general measures $Beta$, see Gnedin \textit{et al.} \cite{MR2896672} page 1442.
\end{itemize} 
\end{example} 
A direct study of the process $(X_t, t\geq 0)$ and its limit based on the SDE (\ref{SDE}) seems a priori rather involved. 
The key tool that will allow us to get some information about $X_{\infty}$ is a duality between $(X_{t}, t\geq 0)$ and a continuous-time Markov chain with values in $\mathbb{N}:=\{1,2,...\}$. Namely consider $(R_{t}, t\geq 0)$ with generator $\mathcal{L}$ defined as follows. For every $g: \mathbb{N}\rightarrow \mathbb{R}$:
\begin{equation} \label{generator} \mathcal{L}g(n)=\sum_{k=2}^{n}\binom{n}{k}\lambda_{n,k}[g(n-k+1)-g(n)]+\alpha n[g(n+1)-g(n)] \end{equation} 
with \[\lambda_{n,k}=\int_{0}^{1}x^{k}(1-x)^{n-k}x^{-2}\Lambda(dx).\]
We have the following duality lemma:
\begin{lemma}\label{dual}For all $x\in [0,1], n\geq 1$,
\[\mathbb{E}[X_{t}^{n}|X_{0}=x]=\mathbb{E}[x^{R_{t}}|R_{0}=n].\]
\end{lemma}
When no selection is taken into account, this duality is well-known (see for instance the recent survey concerning duality methods of Jansen and Kurt \cite{Jansen:arXiv1210.7193}). Several works incorporate selection and study the dual process. We mention for instance the work of Neuhauser and Krone \cite{krone1997anc} in which the Wright-Fisher diffusion case is studied.
For a proof of Lemma \ref{dual}, which relies on standard generator calculations, see Equation 3.11 page 21 in Bah and Pardoux \cite{Bah}.\\

The process $(R_{t}, t\geq 0)$ is clearly irreducible and its properties are related to those of $(X_{t}, t\geq 0)$. The following lemma is crucial in our study. 
\begin{lemma}\label{lemma}
\begin{itemize}
\item[1)] If $(R_{t}, t\geq 0)$ is positive recurrent then the law of $X_{\infty}$ charges both $0$ and $1$.
\item[2)] If $(R_{t}, t\geq 0)$ is transient then $X_{\infty}=0$ almost surely.
\end{itemize}
\end{lemma}
\begin{proof}[Proof of Lemma \ref{lemma}] Recall that $(X_{t}, t\geq 0)$ is positive, bounded and converges almost surely. We first establish 1). Assume that the process $(R_t,t\geq 0)$ is positive recurrent. To conclude that the law of $X_{\infty}$ charges both $0$ and $1$, we use Lemma \ref{dual}. Hence, we have \[\mathbb{P}[X_{\infty}=1|X_{0}=x]=\mathbb{E}[X_{\infty}|X_{0}=x]\geq \mathbb{E}[X_{\infty}^{n}|X_{0}=x]=\mathbb{E}[x^{R_{\infty}}|R_{0}=n]\geq \frac{x^{n_{0}}}{\mathbb{E}_{n_{0}}[T_{n_{0}}]}>0,\]
where $R_{\infty}$ is a random variable with law, the stationary distribution of $(R_{t}, t\geq 0)$ and $T_{n_{0}}$ is the first return time to state $n_0$ of the chain $(R_t, t\geq 0)$. We prove now 2). Assume that the process $(R_t, t\geq 0)$ is transient. Plainly, applying the dominated convergence theorem in Lemma \ref{dual} with $n=1$, we have \[\mathbb{E}[X_{\infty}|X_{0}=x]=\underset{t\rightarrow \infty}\lim \mathbb{E}[x^{R_{t}}|R_{0}=1]=0, \text{ since } R_{t} \underset{t\rightarrow \infty}\longrightarrow \infty \text{ a.s. }\] Thus, $X_\infty=0$ almost surely. \end{proof}

Similarly to the block counting process of a $\Lambda$-coalescent, the process $(R_{t}, t\geq 0)$ has a genealogical interpretation. Roughly speaking, it counts the number of ancestors of a sample of individuals as time goes towards the past. Two kinds of events can occur:
\begin{itemize}
\item[1] A coalescence of lineages. When there are $n$ lineages, it occurs at rate \begin{equation} \label{ratecoal} \phi(n)=\sum_{k=2}^{n}\binom{n}{k}\lambda_{n,k}, \end{equation}
\item[2] A branching (a birth) event (modelling selection). When there are $n$ lineages, the process jumps to $n+1$ at rate $\alpha n$.
\end{itemize}
When a lineage splits in two, this should be understood as two potential ancestors. We refer the reader to Sections 5.2 and 5.4 of \cite{etheridge2011}, and also to Etheridge, Griffiths and Taylor \cite{Etheridge201077} where a dual coalescing-branching process is defined for a general $\Lambda$ mechanism.
\section{Coming down from infinity and study of $(R_{t}, t\geq 0)$}
Rather than working with the process satisfying the SDE (\ref{SDE}), we will work on the continuous-time Markov chain $(R_{t}, t\geq 0)$. Denote $\nu(dx):=x^{-2}\Lambda(dx)$ and define for all $n\geq 2$, \begin{equation} \label{delta} \delta(n):= -n\int_{0}^{1}\log\left(1-\frac{1}{n}[np-1+(1-p)^{n}]\right)\nu(dp).\end{equation} The maps $n\mapsto \delta(n)$ and $n\mapsto \delta(n)/n$ are both non-decreasing and $\delta(n)/n\uparrow \alpha^{\star}.$  For the proof of these monotonicity properties we refer the reader to the proof of Lemma 4.1 and to Corollary 4.2 in \cite{MohleHerriger}.
\\

Firstly, we need to say a word about coalescents and coming down from infinity. Then, we deal with the proof of Theorem \ref{main}. We will adapt some arguments due to M\"ohle and Herriger \cite{MohleHerriger} and use Lemma \ref{lemma}. %In our study, the measures $\Lambda$ that corresponds to $\Lambda$-coalescent that comes down from infinity have to be treated separately. In the first subsection, one has $\alpha=0$.  
\subsection{Revisiting the coming-down from infinity for the $\Lambda$-coalescent}
A nice introduction to the $\Lambda$-coalescent processes is given in Chapter 3 of Berestycki \cite{Beres2}. Denote the number of blocks in a $\Lambda$-coalescent by $(R_{t}, t\geq 0)$. Started from $n$, this process has the generator $\mathcal{L}$, defined in (\ref{generator}), with $\alpha=0$. An interesting property is that this process can start from infinity. We say that the coming down from infinity occurs if almost surely for any time $t>0$, $R_{t}<\infty$, while $R_{0}=\infty$. In this case, $(R_{t}, t\geq 0)$ will be actually absorbed in $1$ in finite time. The arguments that we use to establish Theorem \ref{main} are mostly adapted from technics due to M\"ohle and Herriger \cite{MohleHerriger}. They have established a new condition for $\Xi$-coalescents (meaning coalescents with simultaneous and multiple collisions) to come down from infinity. Their criterion is based on a new function which corresponds to $\delta$ in the particular case of $\Lambda$-coalescents. Their work relies mostly on linear random recurrences. We give here a proof in a "martingale fashion" for the simpler setting of $\Lambda$-coalescents.% and deduce an upper bound for the absorption time of the $\Lambda$-Wright-Fisher process when selection is incorporated. 
\\

The next lemma is lifted from Lemma 4.1 in \cite{MohleHerriger}, however we provide a proof for the sake of completeness. Let $n\geq 2$ and $x\in (0,1)$. We consider the auxiliary random variable $Y_{n}(x)$ with law:
\begin{center} $\mathbb{P}[Y_{n}(x)=l]=1_{l=n}(1-x)^{n}+\binom{n}{l-1}(1-x)^{l-1}x^{n-l+1}$ for every $l\in \{1,...,n\}$. \end{center}
\medskip
\begin{lemma}\label{majorationdelta}
\begin{itemize}
\item[1)] $\mathbb{E}[Y_{n}(x)]=n(1-x)+1-(1-x)^{n}$,
\item[2)] $\frac{\delta(n)}{n}=\int_{0}^{1}-\log \mathbb{E}[Y_{n}(x)/n]\nu(dx)\leq \sum_{j=2}^{n}-\log\left(\frac{n-j+1}{n}\right)\binom{n}{j}\lambda_{n,j}.$
\end{itemize}
\end{lemma}
\begin{proof}[Proof of lemma \ref{majorationdelta}.] The first statement is obtained by binomial calculations and is left to the reader, see Remark 7.2.2 for $\Lambda$-coalescent and Equation (2) in \cite{MohleHerriger}. We focus on the second statement. We have
The first statement is obtained by binomial calculations and is left to the reader, see Remark 7.2.2 for $\Lambda$-coalescent and Equation (2) in \cite{MohleHerriger}. We focus on the second statement. We have
\begin{align*}
\frac{\delta(n)}{n}&=\int_{0}^{1}-\log \mathbb{E}[Y_{n}(x)/n]\nu(dx)\\
&\leq \int_{0}^{1}\mathbb{E}[-\log (Y_{n}(x)/n)]\nu(dx) \text{ by the Jensen inequality } (-\log \text{ is convex})\\
&= \sum_{k=1}^{n-1}-\log\left(\frac{k}{n}\right)\int_{0}^{1}\mathbb{P}[Y_{n}(x)=k]\nu(dx)\\
&= \sum_{k=1}^{n-1}-\log\left(\frac{k}{n}\right)\binom{n}{n-k+1}\lambda_{n,n-k+1}\\
&= \sum_{k=2}^{n}-\log\left(\frac{n-k+1}{n}\right)\binom{n}{k}\lambda_{n,k}.\\
\end{align*}
\end{proof}
\begin{theorem}[M\"ohle, Herriger \cite{MohleHerriger}]\label{CDI} Let $\Lambda$ be a finite measure on $[0,1]$ without mass at $0$. The $\Lambda$-coalescent comes down from infinity if and only if \[\sum_{k\geq 2}\frac{1}{\delta(k)}<\infty.\]
%with $$\delta(n)= -n\int_{0}^{1}\log\left(1-\frac{1}{n}[np-1+(1-p)^{n}]\right)\nu(dp).$$
%The maps $n\mapsto \delta(n)$ and $n\mapsto \delta(n)/n$ are both non-decreasing, and $\delta(n)/n\uparrow \alpha^{\star}.$
Furthermore, we have \[\mathbb{E}[T]\leq 2\sum_{k=2}^{\infty}\frac{1}{\delta(k)},\]
where $T:=\inf\{t\geq 0; R_{t}=1\}$.
\end{theorem}
\begin{proof}[Proof of Theorem \ref{CDI}.] Schweinsberg \cite{CDI} established that a necessary and sufficient condition for the coming down from infinity is the convergence of the series $\sum_{l\geq 2}\frac{1}{\psi(l)}$ where \begin{equation} \label{psi} \psi(l):=\sum_{k=2}^{l}\binom{l}{k}\lambda_{l,k}(k-1)=\int_{0}^{1}[lx-1+(1-x)^{l}]x^{-2}\Lambda(dx). \end{equation} 
We easily observe that for all $n\geq 2$, $\delta(n)\geq \psi(n)$. Therefore the divergence of the series $\sum \frac{1}{\delta(n)}$ entails that of $\sum \frac{1}{\psi(n)}$ and we just have to focus on the sufficient part (for a proof of the necessary part based on martingale arguments, we refer to Section 6 of \cite{coaldist}). 
Assume $\sum \frac{1}{\delta(n)}<\infty$, consider the function \[f(l):=\sum_{k=l+1}^{\infty}\frac{k}{\delta(k)}\log\left(\frac{k}{k-1}\right).\] This function is well defined since $\frac{k}{\delta(k)}\log\left(\frac{k}{k-1}\right)\underset{k\rightarrow \infty}{\sim}1/\delta(k)$.
The generator of the block counting process corresponds to $\mathcal{L}$ with $\alpha=0$, thus we study \[\mathcal{L}f(l)=\sum_{k=2}^{l}\binom{l}{k}\lambda_{l,k}[f(l-k+1)-f(l)].\]
We have
\[f(l-k+1)-f(l)\geq \frac{l}{\delta(l)}\sum_{j=l-k+2}^{l}\log \left(\frac{j}{j-1}\right)=\frac{l}{\delta(l)}[\log(l)-\log(l-k+1)]\] and then
\[\mathcal{L}f(l)\geq \frac{l}{\delta(l)}\sum_{k=2}^{l}\binom{l}{k}\lambda_{l,k}\left[-\log\left(\frac{l-k+1}{l}\right)\right].\]
By Lemma \ref{majorationdelta}, we have
\[\sum_{k=2}^{l}\binom{l}{k}\lambda_{l,k}\left[-\log\left(\frac{l-k+1}{l}\right)\right]\geq \delta(l)/l.\]
We deduce that $\mathcal{L}f(l)\geq 1$ for every $l\geq 2$. Then,
since $f(R_{t})-\int_{0}^{t}\mathcal{L}f(R_{s})ds$ is a martingale, by applying the optional stopping theorem at time $T_{n}\wedge k$ where $T_{n}:=\inf\{t; R_{t}=1\}$ when $R_{0}=n$, we get:
\[\mathbb{E}[f(R_{T_{n}\wedge k})]=f(n)+\mathbb{E}\left[\int_{0}^{T_{n}\wedge k}\mathcal{L}f(R_{s})ds\right]\geq f(n)+\mathbb{E}[T_{n}\wedge k]\]
Letting $k \rightarrow \infty$ and using the fact that $f$ is decreasing, we obtain that
\[\mathbb{E}[T_{n}]\leq f(1)-f(n).\] 
Recall that $T_n \uparrow T$ a.s when $n\rightarrow \infty$. The result follows by the monotone convergence theorem.
\end{proof}
\subsection{Proof of Theorem \ref{main}}
%Consider first the case when $\Lambda$ verifies $\sum_{k=2}^\infty 1/\delta(k)<\infty$. Recall that $\delta(k)/k \underset{k\rightarrow \infty}{\longrightarrow} \alpha^\star$. In that case, by Theorem \ref{CDI}, the associated coalescent comes down from infinity and one has $\alpha^\star=\infty$.
%\\  
%\\
%It remains to establish Theorem 1 when $\sum_{k=2}^\infty 1/\delta(k)=\infty$. 
The proof is based on three lemmatas. The first lemma states that the process is non-explosive. In the two next lemmatas, we look for martingale arguments. We highlight that Lemmatas \ref{nonexplosion}, \ref{functionf} and \ref{martingale} below are valid for $\alpha^{\star}\in (0,\infty]$. By convention, if $\alpha^{\star}=\infty$, then $1/\alpha^{\star}=0$.

\begin{lemma}\label{nonexplosion}
The process $(R_{t},t\geq 0)$ is non-explosive.
\end{lemma}
\begin{proof} We show that the only non-negative bounded solution of $\mathcal{L}f=cf$ for $c>0$ is the trivial solution $f=0$. Reuter's criterion (see e.g. Corollary 2.7.3 in Norris's book \cite{Norris} or \cite{Reuter}) provides that the process is non-explosive. If
$\mathcal{L}f(n)=cf(n)$ then we have
\[\alpha nf(n+1)=cf(n)-\sum_{k=2}^{n}\binom{n}{k}\lambda_{n,k}f(n-k+1)+(\phi(n)+\alpha n)f(n).\]
This yields 
\begin{align*}
f(n+1)-f(n)&=\frac{c}{\alpha n}f(n)-\sum_{k=2}^{n}\frac{\binom{n}{k}\lambda_{n,k}}{\alpha n}f(n-k+1)+\frac{\phi(n)}{\alpha n}f(n)\\
%&= \frac{c+\phi(n)}{\alpha n}f(n)-\sum_{k=2}^{n}\frac{\binom{n}{k}\lambda_{n,k}}{\alpha n}f(n-k+1)\\
&= \frac{cf(n)}{\alpha n}+\sum_{k=2}^{n}\frac{\binom{n}{k}\lambda_{n,k}}{\alpha n}(f(n)-f(n-k+1)).
\end{align*}
If $f(1)=0$ then $f(2)=0$ and by an easy induction $f(i)=0$ for all $i\geq 1$. Therefore, if there exists a positive solution, then necessarily $f(1)>0$. Letting $n=1$ in the last equality provides $f(2)-f(1)>0$ (note that when $n=1$, the sum is empty and $\phi(1)=0$). Assume that $f(n)\geq f(i)$ for all $i\leq n$. The last equality above implies that $f(n+1)-f(n)\geq 0$. By induction, we thus have $f(n+1)\geq f(i)$ for all $i\leq n+1$. Finally, $f$ is non-decreasing and one has
$$f(n)=f(1)+\sum_{k=1}^{n-1}(f(k+1)-f(k))\geq \sum_{k=1}^{n-1}\frac{cf(1)}{\alpha k}$$ then $f$ is unbounded.
\end{proof}
\begin{lemma}\label{functionf} Define the function \[f(l):=\sum_{k=2}^{l}\frac{k}{\delta(k)}\log\left(\frac{k}{k-1}\right).\]
Then, with the generator $\mathcal{L}$ of $(R_{t}, t\geq 0)$ defined in (\ref{generator}), we have for all $l\geq 2$
\[\mathcal{L}f(l)\leq -1+\alpha l/\delta(l).\]
\end{lemma}
\begin{proof}[Proof of Lemma \ref{functionf}] By definition, 
\[\mathcal{L}f(l)=\sum_{k=2}^{l}\binom{l}{k}\lambda_{l,k}[f(l-k+1)-f(l)]+\alpha l[f(l+1)-f(l)].\]
We have $f(l-k+1)-f(l)=-\sum_{j=l-k+2}^{l}\frac{j}{\delta(j)}\log\left(\frac{j}{j-1}\right)$, and since $(j/\delta(j), j\geq 2)$ is decreasing, for all $j\leq l$, $j/\delta(j)\geq l/\delta(l)$. Therefore
\[f(l-k+1)-f(l)\leq -\frac{l}{\delta(l)}\sum_{j=l-k+2}^{l}\log\left(\frac{j}{j-1}\right)=-\frac{l}{\delta(l)}\log\left(\frac{l}{l-k+1}\right).\]
We deduce that
\begin{align*} \mathcal{L}f(l)&\leq -\frac{l}{\delta(l)}\sum_{k=2}^{l}\binom{l}{k}\lambda_{l,k}\log\left(\frac{l}{l-k+1}\right)+\alpha \frac{l+1}{\delta(l+1)} \underbrace{l\log\left(1+\frac{1}{l}\right)}_{\leq 1}\\
&\leq \frac{l}{\delta(l)}\underbrace{\sum_{k=2}^{l}\binom{l}{k}\lambda_{l,k}\log\left(\frac{l-k+1}{l}\right)}_{\leq -\delta(l)/l}+\alpha \frac{l+1}{\delta(l+1)}\\
&\leq -1+\alpha \frac{l}{\delta(l)}.
\end{align*}
The second inequality holds by Lemma \ref{majorationdelta}.
\end{proof}
The following lemma tells us that if $\alpha<\alpha^{\star}\in (0,\infty]$, then $(R_{t}, t\geq 0)$ is positive recurrent. Applying Lemma \ref{lemma} yields the first part of Theorem \ref{main}.  
\begin{lemma}\label{martingale} Assume $\alpha<\alpha^{\star}$. Then, there exists $n_{0}$, such that for all $n\geq n_{0}$, $\mathbb{E}_{n}[T^{n_{0}}]<\infty$, where
\[T^{n_{0}}:= \inf\{s\geq 0; R_{s}<n_{0}\}.\]
Thus, the process $(R_{t}, t\geq 0)$ is positive recurrent. %(and in particular does not go to $\infty$ almost surely).
\end{lemma}
\begin{proof}[Proof of Lemma \ref{martingale}] %Let $N\in \mathbb{N}$ and define $f_{N}(l)=f(l)1_{\{l\leq N+1\}}$. The map $f_N$ is bounded and 
%Recall that $\phi(k)$ was defined in (\ref{ratecoal}). Clearly $\delta(k)\geq \phi(k)$. Moreover one can check that $\sum_{k=2}^\infty \frac{1}{\phi(k)}=\infty$ entails that $\sum_{k=2}^\infty \frac{1}{\phi(k)+\alpha k}=\infty$ (apply for instance Lemma 10 in \cite{CDI} or see Section 6 page 373 of \cite{coaldist}). We deduce that the process $(R_t, t\geq 0)$ is non-explosive. 
For every $N\in \mathbb{N}$, define $$f_{N}(l):=f(l)1_{l\leq N+1}.$$ By Dynkin's formula, the process \[\left(f_{N}(R_{t})-\int_{0}^{t}\mathcal{L}f_{N}(R_{s})ds, t\geq 0\right)\] is a martingale. One can easily check that $\mathcal{L}f_N(l)=\mathcal{L}f(l) \mbox{ if } l\leq N$. For any $\epsilon >0$ there exists $n_{0}$ such that for all $l\geq n_{0}$, 
\begin{equation}\label{epsilon}
\frac{l}{\delta(l)}\leq \frac{1}{\alpha^{\star}}+\epsilon. 
\end{equation}
Let $n_0\leq n \leq N$ and consider the stopping time $S_N:=\inf \{s\geq 0; R(s)\geq N+1\}$.  
We apply the optional stopping theorem to the bounded stopping time $T^{n_{0}}\wedge S_{N}\wedge k$ and obtain
\begin{align*}
\mathbb{E}_{n}[f_{N}(R_{T^{n_{0}}\wedge S_{N}\wedge  k})]&= f_{N}(n) + \mathbb{E}\left[\int_{0}^{T^{n_{0}}\wedge S_{N} \wedge k}\mathcal{L}f_{N}(R_{s})ds\right]\\
&\leq f_{N}(n)+\mathbb{E}\left[\int_{0}^{T^{n_{0}}\wedge S_{N} \wedge k}\left( -1+ \alpha \frac{R_{s}}{\delta(R_{s})}\right)ds\right]\\
&\leq f_{N}(n)+\mathbb{E}\left[\int_{0}^{T^{n_{0}}\wedge S_{N}\wedge k}\left(-1+\alpha(\frac{1}{\alpha^{\star}}+\epsilon) \right)ds\right]\\
&= f_{N}(n)+\left(\frac{\alpha}{\alpha^{\star}}-1+\epsilon \alpha \right)\mathbb{E}[T^{n_{0}}\wedge S_{N} \wedge k].
\end{align*}
The first inequality follows from the equality $\mathcal{L}f_{N}(l)=\mathcal{L}f(l)$ when $l\leq N$ and from Lemma \ref{functionf}. The second inequality follows from (\ref{epsilon}). For small enough $\epsilon$, $1-\frac{\alpha}{\alpha^{\star}}-\epsilon \alpha >0$, thus
\[\underbrace{(1-\frac{\alpha}{\alpha^{\star}}-\epsilon \alpha)}_{>0}\mathbb{E}[T^{n_{0}}\wedge S_{N} \wedge k]\leq f_{N}(n)-\mathbb{E}_{n}[f_{N}(R_{T^{n_{0}}\wedge S_{N}\wedge k})]\leq f_{N}(n),\]
On the one hand, since the process is non-explosive, $S_{N}\underset{N\rightarrow \infty}\longrightarrow \infty$ almost surely and 
therefore, for all $n\geq n_{0}$ \[(1-\frac{\alpha}{\alpha^{\star}}-\epsilon \alpha)\mathbb{E}[T^{n_{0}}\wedge k]\leq f(n).\]
On the other hand, by letting $k\rightarrow \infty$ we get
\begin{center} $\mathbb{E}[T^{n_{0}}]\leq Cf(n)$ for all $n\geq n_{0}$\end{center} 
with $C$ a constant depending only on $\epsilon$. 
\end{proof}

In order to get statement 2) of Theorem \ref{main}, we will apply the second part of Lemma \ref{lemma}. Namely, we show that if $\alpha>\alpha^{\star}$, then $(R_t, t\geq 0)$ is transient.
\begin{lemma} \label{transience} If $\alpha>\alpha^{\star}$ then $R_t \underset{t\rightarrow \infty}{\longrightarrow} \infty$ almost surely.
\end{lemma}
\begin{proof}[Proof of Lemma \ref{transience}] 
If $g$ is a bounded function such that $\mathcal{L}g(n)<0$ for all $n>n_{0}$, the process $(g(R_{t\wedge T^{n_{0}}}), t\geq 0)$ when starting from $n>n_{0}$, is a supermartingale; with $T^{n_{0}}:=\inf\{t>0, R_{t}<n_{0}\}$. Applying the martingale convergence theorem yields that $\mathbb{P}_{n}(T^{n_{0}}<\infty)<1$. Therefore the process $(R_{t}, t\geq 0)$ is not recurrent, and by irreducibility is transient. We show that the function $g(n):=\frac{1}{\log(n+1)}$ fulfills these conditions. One has
$$\mathcal{L}g(n)=\sum_{k=2}^{n}\binom{n}{k}\lambda_{n,k}\left[\frac{1}{\log(n-k+2)}-\frac{1}{\log(n+1)}\right]+\alpha n \left[\frac{1}{\log(n+2)}-\frac{1}{\log(n+1)}\right].$$
On the one hand, one can easily check that
\[\alpha n \left[\frac{1}{\log(n+2)}-\frac{1}{\log(n+1)}\right]= \alpha n \frac{\log\left(\frac{n+1}{n+2}\right)}{\log(n+2)\log(n+1)}=-\alpha\frac{1}{\log(n+2)\log(n+1)}(1+o(1)).\] 
On the other hand, denote by $B_{n}(x)$ a random variable with a binomial law $(n,x)$. We have
\begin{align*}
\mathcal{L}^{0}g(n):=\sum_{k=2}^{n}&\binom{n}{k}\lambda_{n,k}\left[\frac{1}{\log(n-k+2)}- \frac{1}{\log(n+1)}\right]\\
&=\int_{0}^{1}\frac{\Lambda(dx)}{x^{2}}\mathbb{E}\left[\frac{\log\left(\frac{n+1}{n-B_{n}(x)+2}\right)}{\log(n+1)\log(n-B_{n}(x)+2)}\right]\\
&=\frac{1}{\log(n+1)}\int_{0}^{1}\frac{\Lambda(dx)}{x^{2}}\mathbb{E}\left[\frac{-\log\left(1-\frac{B_{n}(x)-1}{n+1}\right)}{\log(n+2)+\log\left(1-\frac{B_{n}(x)}{n+2}\right)}\right].
\end{align*}
The last equality holds true since for all $2 \leq k \leq n$, $\log\left(\frac{n+1}{n-k+2}\right)=-\log \left(1-\frac{k-1}{n+1}\right)$ and $\log(n-k+2)=\log(n+2)+\log\left(1-\frac{k}{n+2}\right)$. Moreover $$|(n+1)x-(B_{n}(x)-1)|\leq |nx-B_{n}(x)|+|x+1|$$ and by Chebyshev's inequality, we have 
\begin{align*}
\mathbb{P}\left[\left| x- \frac{B_{n}(x)-1}{n+1} \right|>(n+1)^{-1/3}\right] &\leq \frac{\text{Var}(B_{n}(x))}{\big((n+1)^{2/3}-(1+x)\big)^{2}}\\
&\leq \frac{nx(1-x)}{\big((n+1)^{2/3}-2\big)^{2}}.   
\end{align*}
Notice that $n\big((n+1)^{2/3}-2\big)^{-2}\underset{n\rightarrow \infty}{\sim}n^{-1/3}$ and $\int_{0}^{1}\frac{\Lambda(dx)}{x^{2}}x(1-x)<\infty$.
Therefore, from the last expression of $\mathcal{L}^{0}g(n)$ above, we have
\[\mathcal{L}^{0}g(n)=\frac{1}{\log(n+1)\log(n+2)}(\alpha^\star+o(1)),\]
and thus, since $\alpha>\alpha^\star$
\[\mathcal{L}g(n)=\frac{1}{\log(n+1)\log(n+2)}\big(\alpha^\star-\alpha+o(1)\big)<0, \text{ for } n \text{ large enough.}\]
%&=\frac{1}{\log(n+1)\log(n+2)}\int_{0}^{1}\frac{\Lambda(dx)}{x^{2}}\mathbb{E}\left[\frac{-\log\left(1-\frac{X_{n}(x)-1}{n+1}\right)}{1+\frac{\log\left(1-\frac{X_{n}(x)}{n+2}\right){\log(n+2)}}}\right]
%Consider that $\alpha\geq\alpha^{\star}$. Let $f:l\mapsto l$, we have \[\mathcal{L}f(l)=-\sum_{k=2}^{l}\binom{l}{k}\lambda_{l,k}(k-1)+\alpha l=-\psi(l)+\alpha l,\]
%where $\psi(k)$ is defined in (\ref{psi}). It is readily checked that $\psi(l)\leq \delta(l)$, moreover the map $l \to \delta(l)/l$ is increasing, thus
%\[\mathcal{L}f(l)\geq -\delta(l)+\alpha l= l(\alpha-\delta(l)/l)>l(\alpha-\alpha^{\star})\geq 0.\]

%Therefore the process $(e^{-(\alpha-\alpha^{\star})t}R_{t},t\geq 0)$ is a positive submartingale. On the one hand, if the process is unbounded then obviously $R_{t}\underset{t\rightarrow \infty}\longrightarrow \infty$. On the other hand, if the process is bounded, then it converges almost surely to a variable which is positive with positive probability. On this event, $R_{t}\underset{t\rightarrow \infty}\longrightarrow \infty$. Actually since the Markov chain is irreducible, we have $R_{t}\underset{t\rightarrow \infty}\longrightarrow \infty$ almost surely.
\end{proof}
\begin{remark}\label{rem}
Bah and Pardoux \cite{Bah} have established (Theorem 4.3) that the absorption of the process $(X_{t}, t\geq 0)$ in finite time is almost sure if and only if the underlying $\Lambda$-coalescent comes down from infinity. Furthermore, Proposition 4.4 in \cite{Bah} states that for all $x\in (0,1)$, $0<\mathbb{P}[X_{\zeta}=0|X_{0}=x]<1$ where $\zeta$ is the absorption time. In such cases, Theorem \ref{CDI} yields plainly that $\alpha*=\infty$. We stress that our arguments still hold when $\Lambda(\{0\})>0$, one only has to add the quadratic term $\Lambda(\{0\})\binom{n}{2}$ to the function $\delta$ (Equation \ref{delta}).
\end{remark}
We end this article by observing a link between the threshold $\alpha^{\star}$ and the first moment of a subordinator.
\begin{remark}
Assume $\alpha^{\star}<\infty$. Then, the corresponding $\Lambda$-coalescent process has dust, meaning that it has infinitely many singleton blocks at any time. As time passes, the asymptotic frequency of the singleton blocks altogether is given by a process $(D(t), t\geq 0)$ with values in $]0,1]$ such that \[(D(t), t\geq 0)=(\exp(-\xi_{t}), t\geq 0)\] where $\xi$ is a subordinator with Laplace exponent \[\phi(q)=\int_{0}^{1}[1-(1-x)^{q}]x^{-2}\Lambda(dx).\] 
We refer the reader to Proposition 26 in Pitman \cite{Pitman}. An interesting feature, easily checked, is that $\alpha^{\star}=\mathbb{E}[\xi_{1}]$. Hence one could expect some fluctuations in $(R_{t},t\geq 0)$ when considering the critical case $\alpha=\alpha^{\star}$. We note, in this context, that R. Griffiths proved in \cite{Griffiths} that $X_{\infty}=0$ almost-surely when $\alpha=\alpha^{\star}$ by a different analytical method.  
\\
\\
Let us also mention that several authors (Gnedin et al. \cite{MR2896672} and Lager{\aa}s \cite{Lageras} for instance) have studied coalescents with a dust component through the theory of regenerative compositions. 
\end{remark}

%%%%%%%%%%%%%%%%%%%%%%%%%%%%%%%%%%%%%%%%%%%%%%%%%%%%%%%%%%%%%%%%%%%
%%                                                               %%
%% You may add acknowledgments (optional).                       %%
%%                                                               %%
%%%%%%%%%%%%%%%%%%%%%%%%%%%%%%%%%%%%%%%%%%%%%%%%%%%%%%%%%%%%%%%%%%%

\ACKNO{I am grateful to colleagues from Goethe Universit\"at Frankfurt  who detected a gap in the original proof of Lemma 2.5 and devised the arguments with the Lyapunov function $g$.
\\
\\
This research was supported by the Research Training Group 1845 of Berlin (German Research Council (DFG)). The author is grateful to Matthias Hammer for fruitful discussions and thanks Jochen Blath and Etienne Pardoux. The author would like to thank the referees for their very careful reading.}

%%%%%%%%%%%%%%%%%%%%%%%%%%%%%%%%%%%%%%%%%%%%%%%%%%%%%%%%%%%%%%%%%%%
%%                                                               %%
%% You have reached the end of your document.                    %%
%%                                                               %%
%%%%%%%%%%%%%%%%%%%%%%%%%%%%%%%%%%%%%%%%%%%%%%%%%%%%%%%%%%%%%%%%%%%

\end{document}